\documentclass[preprint]{amsart}
\usepackage{amsmath}
\usepackage{amssymb}
\usepackage{amsthm}
\usepackage{color}
\usepackage{enumerate}

\DeclareMathOperator{\ab}{ab}

\newcommand{\N}{\mathbb{N}}

\newtheorem{prevtheorem}{Theorem}

\newtheorem{theorem}{Theorem}
\newtheorem{lemma}[theorem]{Lemma}

\newtheorem{conjecture}[theorem]{Conjecture}
\newtheorem{corollary}[theorem]{Corollary}

\theoremstyle{definition}

\newtheorem{example}[theorem]{Example}
\newtheorem{question}[theorem]{Question}

\theoremstyle{remark}
\newtheorem{remark}[theorem]{Remark}

\title{The probability distribution of word maps on finite groups}\date{\today}
\author{William Cocke}
\address{William Cocke\\Department of Mathematics, University of Wisconsin-Madison}
\email{cocke@math.wisc.edu}

\author{Meng-Che ``Turbo" Ho}
\address{Meng-Che ``Turbo" Ho\\Department of Mathematics, Purdue University}
\email{ho140@purdue.edu}

\begin{document}

\maketitle
\begin{abstract}
Word maps provide a wealth of information about finite groups. We examine the connection between the probability distribution induced by a word map and the underlying structure of a finite group. We show that a finite group is nilpotent if and only if every surjective word map has fibers of uniform size. Moreover, we show that probability distributions themselves are sufficient to identify nilpotent groups, and these same distributions can be used to determine abelian groups up to isomorphism. In addition we answer a question of Amit and Vishne.
\end{abstract}

\section{Introduction}

A word $w$ is an element of the free group $w\in \textbf{F}_n=\textbf{F}\langle x_1,\dots,x_n \rangle$. The length of $w$ is the number of variables and their inverses that appear in $w$. For any group $G$, the word $w$ induces a map $w:G^n\rightarrow G$. We write $w$ for both the word $w$ and the word map on $G$, and write $w(G)$ to mean the image $w(G^n)$ of the word map. We will also write $\overline{g}$ to mean the tuple $(g_1,g_2,\cdots,g_n)$. If $G$ is finite then the word map $w$ induces a probability distribution on $G$ where 
\[
\mu_{G,w}(g) = \frac{ \left|\left\{ (\overline{g}\in G^n : w(\overline{g}) = g\right\}\right|}{|G|^n}.
\] 

There has been some interest in examining the relationship between the probability distributions induced by word maps on $G$ and the structure of the underlying finite group $G$; for example, see Ab\'{e}rt \cite{Ab06}. Nikolov and Segal \cite{NS} show that a finite group is solvable if and only there is some $\epsilon>0$ for which $\mu_{G,w}(1)\geq \epsilon$ for all words $w$. Since a finite group admits infinitely many word maps, results about all word maps are in general not computable. However, a finite group admits only finitely many word maps on a fixed number of variables, it is natural to ask what structural information about $G$ can be observed from the collection of probability distributions induced by all $n$-variable word maps on $G$. 

We formalize the above approach as the following question: 

\begin{question}
Fix $n\in \N$, a finite group $G$, and an enumeration of the elements of $G$. Let $g_i$ be the $i$-th element of $G$. Consider the probability distribution of the word map $w$ as a function $f_w: |G| \to \N$ where $f_w(i) = |w^{-1}(g_i)|$. Given the distributions of all $n$-variable word maps of $G$ as a set, what information can be recovered about $G$?
\end{question}

In particular the actual distribution $\mu_{G,w}(g)$ takes as input an element $g$ of the group $G$. The distribution $f_w$ takes as input a natural number, which corresponds to an unknown element $g$ of an unknown group $G$. To the reader uninterested in the semantics of our formalization, it suffices to say that we are interested in distributions that are disconnected from the underlying group. We then ask what information about the group remains. 

In this article we show the following:

\begin{prevtheorem}\label{nilpotent-from-distributions}
For all $n\in \N$, we can identify when a finite group $G$ is nilpotent from the set of distributions of all $n$-variable word maps on $G$.
\end{prevtheorem}

In particular, part of our proof of Theorem \ref{nilpotent-from-distributions} relies on the following theorem, which is certainly of independent interest:

\begin{prevtheorem}\label{uniform-prev}
Let $G$ be a finite group. Then $G$ is nilpotent if and only if for every surjective word map $w$, the distribution $\mu_{G,w}$ is uniform. 
\end{prevtheorem}

The authors are not aware of any two finite groups that have the same set of probability distributions for all $n$. In this direction we also show the following:

\begin{prevtheorem}\label{distribution-abelian}
For $n>1$, the set of distributions of all $n$-variable word maps on $G$ can be used to identify whether $G$ is abelian; moreover, if $G$ is abelian, then the set of distributions identifies $G$ up to isomorphism.  
\end{prevtheorem}

Before giving an outline of the paper, we give the example of all four probability distributions induced by two-variable word maps on the $Q_8$. Note that there are 32 two-variable word maps on $Q_8$, but they induce only 4 probabilty distributions. Moreover, out of the 32 two-variable word maps, there are only 5 automorphism classes of words on $Q_8$.
\begin{align*}
   &\left( \tfrac{1}{8}, \tfrac{1}{8}, \tfrac{1}{8}, \tfrac{1}{8}, \tfrac{1}{8}, \tfrac{1}{8}, \tfrac{1}{8}, \tfrac{1}{8} \right) \\
    &\left(  1, 0, 0, 0, 0, 0, 0, 0 \right)\\
    &\left(  \tfrac{5}{8}, \tfrac{3}{8}, 0, 0, 0, 0, 0, 0 \right)\\
    &\left(  \tfrac{1}{4}, \tfrac{3}{4}, 0, 0, 0, 0, 0, 0 \right).
\end{align*}
A theme of this paper is what type of information can be gained from the sequences of tuples. Note, we are not assuming that the individual elements in the ordering of $G$ are known, nor that one can naturally associate a distribution to a word map, (or for that matter a word map to a word). We do require that the enumeration of $G$ is consistent across the various probability distributions. Clearly the position of the identity element of $G$ can be determined from the ordering. 

There are many open questions regarding the connection between the probability distributions of word maps on a group $G$ and the underlying group $G$. In section 2 we give an answer to a question of Amit and Vishne \cite{Am11}. 

In Section 3 we provide a proof of Theorem \ref{uniform-prev} and in Section 4 we prove Theorems \ref{nilpotent-from-distributions} and \ref{distribution-abelian}.

\section{Probability distributions of word maps}
We will denote the group of all word maps on $d$-variables over a group $G$ as $\textbf{F}_d(G)$. The multiplication is point-wise multiplication, or equivalently, the concatenation of the words inducing the word maps. The groups $\textbf{F}_d(G)$ are very interesting and are examples of reduced free groups. If $G$ is finite then $\textbf{F}_d(G)$ is finite. We will write $\textbf{F}_r$ for the free group of rank $r$. 

Amit and Vishne asked the following two questions \cite{Am11}:

\begin{question}\label{unanswered}
Suppose $N_{w,G}=N_{w',G}$ for every finite group $G$. Does it follow that $w'$ is mapped to $w$ by some automorphism of $\textbf{F}_r$?
\end{question}

\begin{question}\label{answered}
Suppose $N_{w,G}=N_{w',G}$ for a fixed group $G$. Does it follow that $w'$ is mapped to $w$ by some automorphism of $\textbf{F}_r(G)$. 
\end{question}
 
We will answer Question \ref{answered} in the negative below. Work by Puder and Parzanchevski has shown that in certain situations Question \ref{unanswered} is true: if $N_{w,G}$ and $N_{w',G}$ are uniform for all groups $G$, then there is some there is some automorphism of $\textbf{F}_r$ taking $w'$ to $w$ \cite{Pu15}. It is unknown whether Question \ref{unanswered} holds in general. 

Recall that a word $w$ is called a \emph{law on $G$} if it induces the trivial map on $G$. We note the following observation:
\begin{lemma}
Let $G$ be a finite group. A word $w(x,y)=x^k c(x,y)$ where $k\in \mathbb{Z}$ and $c\in F'$, is a law if and only if $x^k$ is a law and $c(x,y)$ is a law.
\end{lemma}

\begin{proof}
Clearly the product of two laws of $G$ is a law of $G$. If $w(x,y)$ is a law of $G$, then $w(x,x)$ is a law of $G$. Hence $x^k$ and consequently $x^{-k}$ are laws of $G$. Therefore $x^{-k} w(x,y) = c(x,y)$ is a law of $G$. 
\end{proof}

The below example is the \textbf{answer in the negative} to Question \ref{answered} of Amit and Vishne:
\begin{example}
There are word maps over $S(3)$ that induce the same probability distribution on $S(3)$ but are not automorphic. Consider the words $w=x^2$ and $v=[x,yx^2 y^2]$. Over $S(3)$ the words $w$ and $v$ induce the same probability distribution. But, they are not automorphic since if $w$ were automorphic to $v$, then $x^2$ would be in the commutator subgroup. Hence over $S(3)$, $x^2c$ would be a law for some $c\in \textbf{F}'$. But, from the above lemma, this is impossible. 
\end{example}

Even restricting to nilpotent groups, the authors have found that there are 5 automorphism classes of word maps over $Q_8$, but only 4 probability distributions over it. 

%\textcolor{red}{
%\begin{proposition}
%In a finite reduced free nilpotent group $G$, two word maps $g$ and $h$ not in the commutator are bi-homomorphic if and only if they are automorphic.
%\end{proposition}
%\begin{proof}
%Suppose the two word maps use $d$ variables. Then by replacing $G$ with $\textbf{F}_d(G)$, we may consider the two word maps as elements in $G$, and we need to show they are automorphic whenever they are bi-homomorphic.\\
%We induct on the nilpotency class of $G$. Suppose first that the nilpotency class of $G$ is 1, i.e.\ $G$ is abelian. 
%\\Example: $x^2$ and $x^2c$ are bi-homomorphic if and only if $c$ is a square.
%\end{proof}
%}

%\begin{example}
%In the Heisenberg group $H(\Z) = \langle a,b \rangle$ with $c = [a,b]$, $a^8c$ is bi-homomorphic to $a^8c^3$, but not automorphic.

%Indeed, they are bi-homomorphic via $\phi$ and $\psi$ where $\phi(a) = a$, $\phi(b) = b^3$, and $\psi(a) = ac^{-1}$, $\psi(b) = b^3$.

%However, if $\tau$ is an automorphism taking $a^8c$ to $a^8c^3$, then we must have $\tau(a) = ac^k$. Thus, we must also have $\tau(b) = b^{\pm1}c^\ell$. Then $\tau(a^8c) = a^8c^{8k\pm 1}$, but $3\neq \pm 1 \pmod 8$, a contradiction.
%\end{example}

\section{Nilpotent groups}\label{nilpotent}
The lemma below was noted by the first author in \cite{Cocke}. It establishes an interesting condition for a group to be nilpotent based solely on a property of the order function. We will use this condition in Theorem \ref{uniform-prev} to construct for a non-nilpotent group $G$ a surjective word $w$ such that $w$ does not induce the uniform distribution on $G$. 

\begin{lemma}\label{qp-lemma}{\cite{Cocke}}
For a prime $p$, a finite group $G$ is not $p$-nilpotent if and only if there are $x,y\in G$, both of order $q^k$ for some prime $q \neq p$ such that the order of $xy$ is $p$ (or the order of $xy$ is either 2 or 4 when $p = 2$.)
\end{lemma}

%\begin{lemma}
%Let $N$ be a finite nilpotent group, and $w(\overline{x};\overline{g})$ is a word with parameters $\overline{g} \in N$. Then $w$ has uniform fiber sizes if and only if the greatest common divisor of the exponents of variables in $\overline{x}$ together with $|N|$ is 1.
%\end{lemma}

%\begin{proof}
%We will induct on the nilpotency class of $N$. Clearly the statement holds if $N$ is abelian. Now suppose that $w$ has uniform fiber sizes for all finite nilpotent groups. Then, if the greatest common divisor of the exponents of variables in $\overline{x}$ together with $|N|$ were $d>1$, then over the cyclic group $C_d$, we have that $w(\overline{x};1,\dots,1)$ is automorphic in the group of word maps on $C_d$ to $x^d$, which does not have uniform size.  

%For the reverse direction, we now note the following: the word $w(\overline{x};\overline{g})$ has uniform fiber sizes on $N$ implies $w(\overline{x};\overline{gN})$ has uniform fiber sizes on $N/\textbf{Z}(N).$ Hence quotienting by the center repeatedly we reach an abelian group and conclude that $w$ has the desired form. 
%\end{proof}

\begin{lemma}\label{us}
Let $N$ be a finite nilpotent group, and $w(\overline{x};\overline{g})$ be a word with parameters $\overline{g} \in N$. Then the following are equivalent:
\begin{enumerate}
\item $w(\overline{x};\overline{g})$ has uniform fiber size over $N$.
\item $w(\overline{x};\overline{g})$ is surjective.
\item The greatest common divisor of the exponents of variables in $\overline{x}$ in $w(\overline{x};\overline{g})$ together with the exponent of $N$ is 1.
\end{enumerate}
\end{lemma}

\begin{proof}
$(1)\to(2)$ is obvious.

$(2)\to(3)$: Suppose (2) holds, but the greatest common divisor of the exponents of variables in $\overline{x}$ in $w(\overline{x};\overline{g})$ together with the exponent of $N$ is $d>1$. Let $p$ be a prime divisor of $d$. Then $p$ divides the exponent of the abelianization $\ab(N)$. In $\ab(N)$, the image of $w(\overline{x};\overline{g})$ is a coset of $w(\overline{x};\overline{1})$. However, if $p$ divides the greatest common divisor of the exponents of $\overline{x}$ in $w$, we have that $w(\ab(N);\overline{1}) \subseteq (\ab(N))^p$, which is strictly smaller than $\ab(N)$ since $p$ divides the exponent of $\ab(N)$.

$(3)\to(1)$: Suppose (3) holds, then without loss of generality, we may assume the word has the form $ x_1hc$ where $h$ is a word in $\overline{g}$ and $c$ is a commutator word in the variables $\overline{x}$ and parameters $\overline{g}$. It is clear that this has uniform fiber size over an abelian group. 

We now induct on the nilpotency class of $N$. Let $Z(N)$ be the center of $Z$. By the induction hypothesis, $w(\overline{x};\overline{g})$ has uniform fiber size over $N/Z(N)$, after replacing the parameters by their canonical image. Thus, it suffices to show that for every $a,b \in N$ such that $a^{-1}b\in Z(N)$, $w$ has the same fiber size over $a$ and $b$. However, we have the bijection $(x_1,x_2,x_3,\cdots) \to (x_1a^{-1}b,x_2,x_3,\cdots)$ between the fibers of $a$ and $b$. Indeed, suppose that $w(x_1,x_2,\cdots) = a$. As $a^{-1}b$ is in the center and $c$ is a commutator word, we have $c(x_1a^{-1}b,x_2,\cdots) = c(x_1,x_2,\cdots)$, thus $w(x_1a^{-1}b,x_2,\cdots) = x_1a^{-1}bhc(x_1a^{-1}b,x_2,\cdots) = (a^{-1}b)(x_1hc(x_1,x_2,\cdots)) = a^{-1}ba = b$. The other implication can be established similarly. So $w$ has uniform fiber size over $N$, completing the proof.
\end{proof}

We now prove a slightly stronger version of Theorem \ref{uniform-prev}.
\begin{theorem}\label{uniform-theorem}
Let $G$ be a finite group. Then the following are equivalent:
\begin{enumerate}
\item $G$ is nilpotent.
\item For every surjective word map $w$, the distribution $\mu_{G,w}$ is uniform.
\item There is some $n>1$ such that for every $n$-variable surjective word map $w$, the distribution $\mu_{G,w}$ is uniform.
\end{enumerate}
\end{theorem}

\begin{proof}

$(1) \to (2)$ We first suppose that $G$ is nilpotent. Then by the previous lemma, if a word map is surjective, then is has uniform fiber size.

$(2) \to (3)$ is obvious.

$(3) \to (1)$ Now suppose that $n > 1$ and every $n$-variable surjective word map on $G$ induces the uniform distribution. We will show $G$ is $p$-nilpotent for every prime $p$. Suppose by way of contradiction that $G$ is not $p$-nilpotent for the prime $p$. Then by Lemma \ref{qp-lemma} above, there are two elements $a,b$ of $G$, such that $o(a)=o(b)=q^k$ and \[  o(ab) \begin{cases}= p & \text{for $p$ an odd prime}\\  \in\{2,4\} & \text{for $p$=2.}\end{cases}\]  Since $p$ and $q$ are coprime there are $r,s \in \mathbb{Z}$ such that $rp+sq^k=1$; (in the event $o(ab)=4$ we will assume that $4r+sq^k=1$). Consider the $n$-variable word \[w(\overline{x})=x_1^{sq^k} x_2^{sq^k} (x_1x_2)^{rp},\] (if necessary let $p=4$). We have the following facts about $w$:

\begin{enumerate}[(a)]
\item For any $g\in G$, we have $w(\overline{x})=g$ if $x_1 = g$ and $x_2 = 1$. 
\item For any $g\in G$, we have $w(\overline{x})=1$ if $x_1 = g$ and $x_2 = g^{-1}$. 
\item If $x_1 = a$ and $x_2 = b$, we have \[w(a,b)=a^{sq^k} b^{sq^k} (ab)^{rp} = 1.\]
\end{enumerate}

By (a), $w$ is surjective. By (b) and (c), there are at least $(|G|+1)\cdot|G|^{n-2}$ tuples in $G^n$ that map to the identity. So $w$ is a surjective word map on $G$ that does not induce the uniform distribution. We conclude that if every $n$-variable surjective word map on $G$ induces the uniform distribution, then $G$ is $p$-nilpotent for every prime $p$, and hence nilpotent.
\end{proof}

\begin{corollary}
There are two word maps over a reduced free group that are homomorphic images of each other, but not related via an automorphism. 
\end{corollary}

\begin{proof}
Consider a non-nilpotent reduced free group $G$. By the previous theorem, there is a surjective word map $w$ on $G$ such that it does not induce the uniform distribution. However, as it is surjective, there is a homomorphism of $G$ taking $w$ to $x$, and vice versa. Thus $w$ and $x$ are homomorphic images of each other, but not automorphic, since $x$ induces the uniform distribution.
\end{proof}

\begin{corollary}
In a finite nilpotent group $G$, the equation $x=c(x,y)$ where $c\in F'[x,y]$ has exactly $|G|$ solutions; moreover the solution set is exactly the two-tuples in the set $\{(1,g):g\in G$\}.
\end{corollary}

\begin{proof}
We note that a solution $(a,b)$ to $x=c(x,y)$ is also a solution to $w(x,y)=1$ where $w=x^{-1}c(x,y)$. Since $G$ is nilpotent and $w$ is surjective, we see that there are exactly $|G|$ such solutions. Clearly, $(1,g)$ is a solution for all $g\in G$. 
\end{proof}

The above corollary can easily be generalized to the following:
\begin{corollary}\label{not-Alon}
In a finite nilpotent group $G$, the equation $w(\overline{x})=c(\overline{x}),$ where $c\in \textbf{F}'(G)[\overline{x}]$ and $w$ is a surjective word map on $G,$ has exactly $|G|^{|\overline{x}|-1}$ solutions. 
\end{corollary}

\begin{proof}
By Theorem \ref{uniform-theorem}, there is an automorphism $\sigma$ of $G$ that takes $w$ to $x_1$. Then $\sigma(w^{-1}c)=x_1^{-1}\sigma(c)$ and we conclude that $x_1^{-1}\sigma(c)$ is a surjective word map on $G$. Hence, there are exactly $|G|^{|\overline{x}|-1}$ solutions to $x_1^{-1}\sigma(c)$ and these are in bijection (via $\sigma$) with the solutions to $w(\overline{x})=c(\overline{x})$. 
\end{proof}

It is natural to ask about the generalization of Corollary \ref{not-Alon} to the equation $w(\overline{x})=v(\overline{x})$, without any restriction on $w$ or $v$, which is equivalent to considering the equation $w(\overline{x})=1$. Amit conjectured the following in unpublished work \cite{Am} (see also \cite[Question 24]{Ni11}):

\begin{conjecture}[Amit]
For every word map $w(\overline{x})$ on a finite nilpotent group $G$, \[\mu_{G,w}(1)\geq \frac{1}{|G|},\] i.e., the number of solutions to $w(\overline{x}) = 1$ is greater than or equal to $|G|^{|\overline{x}|-1}$.
\end{conjecture}

There are only some partial results in this direction. Levy has shown that when $G$ has nilpotent class 2, then for any word $w$ we have $\mu_{G,w}(1)\geq \frac{1}{|G|}$ \cite{Levy}; showing that Amit's Conjecture holds for class 2 groups. I{\~{n}}iguez and Sangroniz have shown the stronger condition that for free $p$-groups of nilpotency class 2 and exponent 2, it is true that $\mu_{G,w}(g)\geq \frac{1}{|G|}$ \cite{IS}. Klyachko and Mkrtchyan \cite{Kl14} considered first-order formula in any finite group, which implies $\mu_{G,w}(1)\geq \frac{1}{|G|}$ when $w$ has only 2 variables.

\section{Information content of distributions}

In this section we are interested in understanding the information content of the distributions of word maps of a group. Recall that we are interested in the following question:

\begin{question}
Fix $n\in \N$, a finite group $G$, and an enumeration of the elements of $G$. Let $g_i$ be the $i$-th element of $G$. Consider the probability distribution of the word map $w$ as a function $f_w: |G| \to \N$ where $f_w(i) = |w^{-1}(g_i)|$. Given the distributions of all $n$-variable word maps of $G$ as a set, what information can be recovered about $G$?
\end{question}

A priori, the answer of the question depends on $n$. We ask:

\begin{question}
Do we get more information as $n$ gets larger?
\end{question}

From the distributions of word maps we can easily read off the size of the group. Moreover, we can identify the identity element in $G$ as it is the image of the only word map (the identity map) that has an image of size 1.

We mention the following example:
\begin{example}
$D_8$ and $Q_8$ have the same reduced free group on two variables, i.e., \[\textbf{F}_2(D_8)=\textbf{F}_2(Q_8)=\text{SmallGroup(32,2)}.\] However using Magma \cite{Magma} we find that they have different sets of distributions of 2-variable word maps.
\end{example}

\subsection{$n = 1$: words with a single variable}

When $n=1$, the images of the word maps are exactly the sets $G^k = \{g^k \mid g\in G\}$.

\begin{example}\label{n=1ex}
The distribution of 1-variable word maps does not determine even nilpotent groups up to isomorphism. Consider the extra-special groups of exponent $p$. When looking at word maps on $1$-variable, they are indistinguishable from the elementary abelian groups of the same order. 
\end{example}

In order to prove Theorem \ref{nilpotent-from-distributions} we will show that the distribution of $1$-variable word maps on $G$ determines whether or not $G$ is nilpotent; to do this we need the following theorem originally conjectured by Frobenius and proven by Iliyori and Yamaki \cite{IY}:

\begin{theorem}[Frobenius Solutions Theorem]
Let $G$ be a finite group and let $d$ divide the exponenet of $G$. Let $X(d)=\{x\in G: x^d =1\}.$ Then $d$ divides $|X(d)|$ and if $|X(d)|=d$ then $X(d)$ is a normal subgroup of $G$. 
\end{theorem}

Using the Frobenius Solutions Theorem we can now show:

\begin{theorem}\label{n=1nil}
The distribution of 1-variable word maps on a finite group $G$ determines whether or not $G$ is nilpotent. 
\end{theorem}

\begin{proof}
We first note, that the identity element is always determined by the set of distributions, i.e., the only element for which there is a distribution mapping entirely onto it. Moreover, we can always identify $|G|$ from the set of distributions, namely by looking for a uniform distribution. 

 Let $|G|=p^km$ where $\gcd(p,m)=1$ and $k \ge 1$. Then, if $G$ is nilpotent, there are exactly $p^k$ solutions to the equation $x^{p^k}=1$. Moreover, letting $w=x^{p^k}$, we see that for every $g\in w(G)$ there are exactly $p^k$ preimages in $G$ and $|w(G)|=m$. 

Now suppose $G$ is a group of order $p^km$, where $(p,m)=1$, and $w=x^d$ such that the following hold:
\begin{itemize}
\item For every $g\in w(G)$ there are exactly $p^k$ preimages in $G$.
\item $|w(G)|=m$.
\end{itemize}
We claim $G$ must be nilpotent. We first note that by the Frobenius Solutions Theorem the number $d$ is a $p$-th power. 

Let $X(m)$ be the solutions to the equation $x^m = 1$. Also by the Frobenius Solutions Theorem, $|X(m)| \ge m$. But every element of $X(m)$ is a solution to $x^m = 1$, hence they have order coprime to $p$. Since $w=x^{p^j}$ for some $j\ge 1$ and $\gcd(p^j,m) = 1$, the elements in $X(m)$ must also be in $w(G)$. But, $|w(G)|=m$ and we conclude that $w(G) = X(m)$ and contains no elements of order $p$. Hence every element whose order is a power of $p$ is a solution to $w$. Thus, $G$ has a normal Sylow $p$-subgroup. Then $G$ is nilpotent if and only if there is such a $w$ for all $p$ dividing $|G|$. 
\end{proof}

\subsection{$n >1$: words with more than one variable}

From Theorem \ref{uniform-prev}, we see that for $n>1$, the set of all distributions of $n$-variable word maps on $G$ is enough to determine whether or not $G$ is nilpotent, i.e., a finite group $G$ is not nilpotent if and only if there is some $n$-variable surjective map that is not uniform. This, together with Theorem \ref{n=1nil}, proves Theorem \ref{nilpotent-from-distributions}, which we restate here:

\begin{theorem}
For all $n\in \N$, we can identify when a finite group $G$ is nilpotent from the set of distributions of all $n$-variable word maps on $G$.
\end{theorem}

Interestingly enough the set of all distributions of $n$-variable word maps can also identify abelianess:

\begin{lemma}
For any $n>1$, a finite group $G$ is abelian if and only if the distribution of every $n$-variable word map is uniform over its image.
\end{lemma}

\begin{proof}
If $G$ is abelian, then for every word map $w$, $w^{-1}(0)$ is a subgroup of $G^k$, and $w^{-1}(g)$ is either a coset of it or empty. Thus every word map is uniform over its image.

If $G$ is not abelian, then as shown by Ashurst, for $w=[x,y]$, we have that $\mu_{G,w}(1)>\mu_{G,w}(g)$ for all $g\in G$ \cite[Lemma 2.2.8]{As12}. Also, as $G$ is not abelian, $\mu_{G,w}(g)$ is not all zero for $g\neq 1$. If we regard $w$ as an $n$-variable word, then $\mu_{G,w}$ is not uniform over its image. 
\end{proof}

We will use the following Theorem from Cocke and Jensen \cite{CJ}:
\begin{theorem}\label{CJ}
Let $G$ and $H$ be finite abelian groups. Let $X$ be the set of natural numbers $r$ such that there exists a $k$ so that the word $w=x^k$ satisfies $|G|-|w(G)|=r$. Let $Y$ be the set of natural numbers $s$ such that there exists a $k$ so that the word $w=x^k$ satisfies $|H|-|w(H)|=s$. If $X=Y$, then $G\cong H$. 
\end{theorem}

We now restate and prove Theorem \ref{distribution-abelian}:

\begin{theorem}
For $n>1$, the set of distributions of all $n$-variable word maps on $G$ can be used to identify whether $G$ is abelian; moreover, if $G$ is abelian, then the set of distributions identifies $G$ up to isomorphism.  
\end{theorem}

\begin{proof}
In an abelian group, every word is automorphic to a power word, as can be seen by using a series of Nielson transformation to cancel out all but a single variable. Since the number of not $k$-powers in $G$ is determined by the word map $x^k$, the set of distributions of word maps on $G$ for any number of variables determines the set of natural numbers $m$ such that there exists a $k$ so that the word $w=x^k$ satisfies $|G|-|w(G)|=m$. If we are looking at all distributions induced by $n$-variable word maps where $n>1$, then we can determine if $G$ is abelian. If $G$ is abelian, then by Theorem \ref{CJ} above we have determined $G$ up to isomorphism. 
\end{proof}

\begin{remark}
Note that as shown in Example \ref{n=1ex}, the distribution of 1-variable word maps is not enough to identify whether a group is abelian. However, if in addition to knowing that the distribution of 1-variable word maps, we also assume that $G$ is abelian, then Theorem \ref{CJ} applies and we can still identify $G$ up to isomorphism.
\end{remark}

The reduced free group of a nilpotent group is the direct product of the reduced free groups of its Sylow subgroups \cite[p.~41]{Ne67}. We show that without the group structure, the distributions of word maps of a nilpotent group determine the distributions of word maps of its Sylow subgroups, and similarly the distributions of word maps of all the Sylow subgroups determine the distributions of word maps of $G$.

\begin{theorem}
The distributions of word maps of a nilpotent group uniquely determines the distributions of word maps of its Sylow $p$-subgroups for all $p$, and vice versa.
\end{theorem}

\begin{proof}
Given the distributions of word maps of a finite group $G$, we will show how to identify the sub-lists of the enumerated list $G$ that correspond to the Sylow $p$-subgroups. First, we note that from the distributions of word maps of $G$ we can determine the order of $G$. Write  $|G| = p^nk$ such that $p \nmid k$. Then the word map $x^k$ is uniformly distributed on its image, the Sylow $p$-subgroup. Since $G$ is nilpotent, we have $G = PK$ where $|P| = p^n$ and $|K| = k$. Then the following holds for every word map $w$ and $g\in P$, $h\in K$:
$$w^{-1}_G(gh) = w^{-1}_P(g) w^{-1}_K(h).$$
Suppose $w$ have image of size $p^n$ and is uniform. Then it is uniform when projected to both $P$ and $K$. However, this means that the size of the image in $K$ must divide $|K|^2 = k^2$. But $k$ and $p$ are co-prime, so the size of the image of $w$ in $K$ is 1. Thus, the image of $w$ in $G$ must be $P$. This allows us to identify the Sylow $p$-subgroups. For every word map, we may find its distribution as a word map on $P$ by looking at its distribution on $P$ and scale accordingly. 

For the backward direction, if we enumerate elements in $G$ as the Cartesian product of the elements in the Sylow $p$-subgroups, then from the above discussion we have that any distribution on $G$ is a product of distributions on the Sylow subgroups in the sense that 
$$w^{-1}_G(g_1\cdots g_\ell) = w^{-1}_{P_1}(g_1) \cdots w^{-1}_{P_\ell}(g_{\ell}).$$
Thus, we only need to show that the products of distributions on the Sylow subgroups are actually realized as the distribution of some word map on $G$. Let $P$ be a Sylow $p$-subgroup, and $K$ be its complement. Suppose $|P| = p^n$ and $|K| = k$, and $rp^n+sk = 1$. Then for any word $w$, define $\hat{w}_p(x_1,x_2,\cdots) = w(x_1^{sk},x_2^{sk},\cdots)$. Then $w = \hat{w}_p$ in $P$, and $\hat{w}_p$ is a law on $K$. Thus, the product of distributions on the Sylow subgroups are realized by the product of the $\hat{w}_p$ as a word map in $G$.
\end{proof}

\section{acknowledgments}
This material is based upon work done while the first author was supported by the National Science Foundation under Grant No. DMS-1502553. Some of the work was done while the first author was visiting the Army Cyber Institute. The views expressed are those of the author and do not reflect the official policy or position of the Army Cyber Institute, West Point, the Department of the Army, the Department of Defense, or the US Government.

\bibliographystyle{amsalpha}
\bibliography{ref}

\end{document}